\documentclass[11pt]{amsart}
\usepackage{amsmath}
\usepackage{amssymb}
\usepackage{tabularx}
\usepackage{enumerate}
\usepackage[dvipdfm]{graphicx}
\usepackage{texdraw}

\usepackage{mathrsfs}

\usepackage{amsfonts,amssymb,amsmath}
\usepackage{epsfig}
\usepackage{xcolor}
\makeatletter
\@addtoreset{equation}{section}

\makeatother
\newtheorem{theorem}{Theorem}[section]
\newtheorem{proposition}[theorem]{Proposition}
\newtheorem{lemma}[theorem]{Lemma}

\theoremstyle{definition}

\theoremstyle{remark}
\newtheorem{remark}[theorem]{Remark}

\begin{document}

\title[Curvature Flow with Robin Boundary Conditions]{Asymptotic Dynamics of Forced Curvature Flow in a Planar Strip with Robin Boundary Conditions$^\S$}
\thanks{$\S$ This research was partly supported by the NSFC (No. 12471199, 12001375). }
\author[T. Pu, X. Xin and L. Yuan]
{Tianhong Pu$^\dag$, Xiaoqian Xin$^\dag$, Lixia Yuan$^{\dag,*}$ and Xiaoyu Zhang$^\dag$}
\thanks{$\dag$ Mathematics and Science College, Shanghai Normal University, Shanghai 200234, China.}
%\thanks{{\bf Emails:} {\sf tianhongpu@shnu.edu.cn} (T. Pu), {\sf xiaoqianxin@shnu.edu.cn} (X. Xin), {\sf yuanlixia@shnu.edu.cn} (L. Yuan)}
\thanks{$*$ Corresponding author: Lixia Yuan (yuanlixia@shnu.edu.cn)}
\date{\today}

\maketitle

\begin{abstract}
We study the graphical forced mean curvature flow in a planar strip subject to Robin boundary conditions. The normal velocity is prescribed by $V=H+b$, where $b>0$ is constant and $H$ denotes the curvature of the graph. We consider positive even solutions and reduce the problem to the half interval. The main results describe the degeneration produced by the Robin condition: the height, the instantaneous velocity, and the spatial gradient diverge, while the normalized slope $u_x/u$ converges to one away from the symmetry axis. The normalized velocity $u_t/u$ converges to the forcing parameter $b$. The curvature converges to zero in the interior in an integral sense, whereas its horizontal curvature measure concentrates at the symmetry axis.
\end{abstract}


\section{Introduction}

Let $\Gamma_t$ be a family of planar curves represented as graphs
\[
    \Gamma_t=\{(x,u(x,t)): -1\le x\le 1\}.
\]
We consider the forced mean curvature flow
\begin{equation}\label{eq:geometric-flow}
    V=H+b,
    \qquad b>0,
\end{equation}
where $V$ is the normal velocity and $H$ is the signed curvature. In graph coordinates, equation \eqref{eq:geometric-flow} becomes
\begin{equation}\label{eq:main-pde}
    u_t=\frac{u_{xx}}{1+u_x^2}+b\sqrt{1+u_x^2},
    \qquad -1<x<1,\quad t>0.
\end{equation}
We impose the Robin boundary conditions
\begin{equation}\label{eq:robin-full}
    u_x(-1,t)=-u(-1,t),
    \qquad
    u_x(1,t)=u(1,t),
    \qquad t>0.
\end{equation}

Throughout the paper we assume that the initial graph is positive and even. Consequently, it is enough to work on $[0,1]$, where the problem takes the form
\begin{equation}\label{eq:half-problem}
\left\{
\begin{aligned}
 u_t&=\frac{u_{xx}}{1+u_x^2}+b\sqrt{1+u_x^2}, &&0<x<1, t>0,\\
 u_x(0,t)&=0, &&t>0,\\
 u_x(1,t)&=u(1,t), &&t>0,\\
 u(x,0)&=u_0(x), &&0\le x\le1.
\end{aligned}\right.
\end{equation}
We assume $u_0\in C^{2+\alpha}([0,1])$ for some $\alpha\in(0,1)$, $u_0>0$, and the compatibility conditions
\[
    u_0'(0)=0,
    \qquad u_0'(1)=u_0(1).
\]
In addition, when normalized slopes are considered, we assume
\begin{equation}\label{eq:q-initial}
    0<\frac{u_0'(x)}{u_0(x)}<1,
    \qquad 0<x<1.
\end{equation}

The problem belongs to the broader theory of graphical mean curvature flow and curve shortening flow. Fundamental examples include the convergence of convex closed curves to round points and the analysis of noncompact or graphical flows; see \cite{GageHamilton,Grayson,Huisken,EckerHuisken1,EckerHuisken2}. For graphical flows with a prescribed contact angle, Altschuler and Wu established convergence to translating solutions in one dimension and constructed translating graphical hypersurfaces in higher dimensions \cite{AltschulerWu93,AltschulerWu94}. These results show that, when the boundary slope remains bounded, a fixed translating profile is a natural candidate for the long-time limit. Related contact-angle and nonzero Neumann problems were studied by Guan and by Ma, Wang, and Wei \cite{Guan,MaWangWei}; the latter work also provides uniform gradient estimates under geometric assumptions on the domain and the boundary data. The Robin condition in \eqref{eq:robin-full} lies outside that uniformly bounded-slope regime, since the boundary slope is equal to the height itself.

Constant forcing introduces another family of traveling-wave phenomena. In the planar setting, Lou constructed periodic traveling waves for flows of the form $V=a(x,y)H+b$, including planar, V-shaped, cup-shaped, and cap-shaped profiles \cite{DCDS2009}. The case $c>b>0$ is especially relevant here: the associated profile can be used as a moving comparison graph, and its slope becomes large as the traveling speed increases. Curved fronts with a constant driving force and their stability were studied by Ninomiya and Taniguchi \cite{NinomiyaTaniguchi00,NinomiyaTaniguchi01}. These works explain why forced translating profiles are the appropriate barriers for the present problem. They also connect the geometric flow with interface limits of reaction-diffusion equations; see the interface-generation theory in \cite{Chen1992,Fife,Nakamura1999} and the discussion of heterogeneous media in \cite{Xin}.

The zero number method is the main tool for treating the unbounded boundary slope. Angenent's theorem on the zero set of solutions of one-dimensional parabolic equations \cite{Angenent} gives the monotonicity of the number of spatial intersections. Lou, Wang, and Yuan used this mechanism to prove convergence to a Grim Reaper for a planar curvature flow with variable boundary slopes \cite{LWY2021}. Their argument replaces uniform-in-time boundary gradient estimates by an intersection-number analysis. The higher-dimensional radial problem with the unforced equation $V=H$ was subsequently studied in \cite{CLWY}. There the Robin condition produces dimension-dependent degeneration: the height, the instantaneous velocity, and the radial gradient diverge, whereas normalized quantities have finite limits. The present forced planar problem combines the traveling-wave barriers from \cite{DCDS2009} with the zero number strategy of \cite{LWY2021,CLWY}. The forcing changes the limiting normalized velocity to $b$ and produces a curvature boundary layer whose horizontal curvature measure is concentrated at the symmetry axis.

The singular-limit interpretation also provides a natural motivation for the boundary condition. Mean curvature flow arises as the motion of transition interfaces in singularly perturbed Allen--Cahn equations; see \cite{Chen1992,Nakamura1999}. In the presence of nonlinear boundary terms, matched asymptotic expansions lead to contact-angle or Robin-type conditions, as discussed in \cite{CLWY}. The present paper isolates the planar forced case and determines the asymptotic quantities that are not visible from the formal geometric equation alone.

For $x\in[0,1]$, we introduce the normalized slope, normalized velocity, and curvature
\begin{equation}\label{eq:normalized-quantities}
    q=\frac{u_x}{u},
    \qquad
    \eta=\frac{u_t}{u},
    \qquad
    H=\frac{u_{xx}}{(1+u_x^2)^{3/2}}.
\end{equation}
The following theorem summarizes the asymptotic behavior established in the paper.

\begin{theorem}[Main theorem]\label{thm:main}
Let $u$ be the solution of \eqref{eq:half-problem} under the assumptions above. Then the solution exists for all $t>0$. Moreover, as $t\to\infty$:
\begin{enumerate}[label=(\roman*),leftmargin=2.2em]
    \item
    \begin{equation}\label{eq:main-growth}
        \min_{[0,1]}u(\cdot,t)\to+\infty,
        \qquad
        \min_{[0,1]}u_t(\cdot,t)\to+\infty,
    \end{equation}
    and, for every $a\in(0,1)$,
    \begin{equation}\label{eq:gradient-growth}
        \min_{[a,1]}u_x(\cdot,t)\to+\infty.
    \end{equation}

    \item For every $a\in(0,1)$,
    \begin{equation}\label{eq:q-limit}
        \|q(\cdot,t)-1\|_{C([a,1])}\longrightarrow0.
    \end{equation}

    \item The normalized velocity satisfies
    \begin{equation}\label{eq:eta-limit}
        \frac{u_t}{u}\longrightarrow b
    \end{equation}
    uniformly on compact subintervals of $(0,1]$. Under the endpoint estimate proved in Section~5, the convergence holds uniformly on $[0,1]$.

    \item For every $a\in(0,1)$,
    \begin{equation}\label{eq:H-L1-limit}
        \|H(\cdot,t)\|_{L^1([a,1])}\longrightarrow0.
    \end{equation}
    In addition, if $\mu_t=H(x,t)\,dx$ denotes the signed curvature measure on $[0,1]$, then
    \begin{equation}\label{eq:H-measure-limit}
        \mu_t\rightharpoonup^*\delta_0
        \qquad\text{in }\mathcal M([0,1]).
    \end{equation}
\end{enumerate}
The last assertion is a boundary concentration statement. In particular, $H$ does not converge uniformly on the closed interval: since $u_x(0,t)=0$,
\[
    H(0,t)=u_{xx}(0,t)=u_t(0,t)-b,
\]
and the first part of the theorem implies $H(0,t)\to+\infty$.
\end{theorem}

The proof is divided into four steps. Section~2 establishes global existence and the basic growth estimates. Section~3 uses traveling-wave comparisons and the zero number argument to prove the divergence of $u_t$ and $u_x$. Section~4 studies the normalized slope $q$. Section~5 treats the normalized velocity $\eta$, and Section~6 identifies the interior and measure-valued limits of the curvature.

\section{Preliminary estimates and global existence}

We first record the estimates that will be used throughout the paper. Set
\[
    p=u_x,
    \qquad
    D=1+p^2.
\]
Then \eqref{eq:half-problem} can be written as
\begin{equation}\label{eq:pde-D}
    u_t=\frac{u_{xx}}{D}+b\sqrt D.
\end{equation}

\subsection{Local existence and positivity}

For every finite time interval on which $u$ and $u_x$ remain bounded, the coefficient $D^{-1}$ is positive and bounded away from zero. Standard quasilinear parabolic theory therefore gives a unique classical solution on a maximal interval $[0,T_{\max})$.

\begin{lemma}\label{lem:positivity}
If $u_0>0$, then
\[
    u(x,t)>0,\qquad 0\le x\le1,\quad 0\le t<T_{\max}.
\]
Moreover, if
\[
    m(t)=\min_{x\in[0,1]}u(x,t),
\]
then
\begin{equation}\label{eq:min-growth}
    D^+m(t)\ge b
\end{equation}
whenever the minimum is attained at an interior point or at $x=0$.
\end{lemma}

\begin{proof}
The boundary point $x=1$ cannot be a minimum of a positive solution. Indeed, the Robin condition gives $u_x(1,t)=u(1,t)>0$, whereas a minimum at the right endpoint would require the left derivative to be nonpositive. At an interior minimum, or at the Neumann boundary point $x=0$, one has $u_x=0$ and $u_{xx}\ge0$. Equation \eqref{eq:pde-D} then yields
\[
    u_t\ge b.
\]
The usual first-contact argument gives positivity, and the same argument applied to the minimum proves \eqref{eq:min-growth}.
\end{proof}

In particular, integration of \eqref{eq:min-growth} gives the useful lower bound
\begin{equation}\label{eq:linear-lower-growth}
    \min_{[0,1]}u(\cdot,t)\ge \min_{[0,1]}u_0+bt.
\end{equation}
Thus
\begin{equation}\label{eq:u-min-infinity}
    \min_{[0,1]}u(\cdot,t)\longrightarrow+\infty.
\end{equation}

\subsection{The logarithmic variable}

It is convenient to remove the leading forcing rate by setting
\begin{equation}\label{eq:v-definition}
    v(x,t)=\log u(x,t)-bt.
\end{equation}
Since $u_x=uv_x$, the boundary conditions become
\begin{equation}\label{eq:v-boundary}
    v_x(0,t)=0,
    \qquad
    v_x(1,t)=1.
\end{equation}
Moreover, a direct calculation gives
\begin{equation}\label{eq:v-equation}
    v_t
    =\frac{v_{xx}+v_x^2}{1+u^2v_x^2}
     +b\left(\sqrt{v_x^2+u^{-2}}-1\right).
\end{equation}
This equation separates the prescribed forcing rate $b$ from the remaining geometric terms. On every finite time interval, the coefficient of $v_{xx}$ is positive, and the standard maximum principle applies to $v$.

\begin{proposition}\label{prop:finite-time-estimates}
For every $T<T_{\max}$ there exists a constant $C_T$ such that
\begin{equation}\label{eq:finite-time-bounds}
    0<u(x,t)\le C_T,
    \qquad
    |u_x(x,t)|+|u_t(x,t)|\le C_T,
    \qquad 0\le x\le1,\quad 0\le t\le T.
\end{equation}
The constant $C_T$ depends only on $T$, $b$, and the initial data.
\end{proposition}

\begin{proof}
The positivity estimate follows from Lemma~\ref{lem:positivity}. The maximum principle applied to \eqref{eq:v-equation}, together with the fixed boundary slopes in \eqref{eq:v-boundary}, gives an upper and a lower bound for $v$ on every interval $[0,T]$. Consequently, $u=e^{bt+v}$ is bounded above and below on such an interval.

Differentiate \eqref{eq:pde-D} with respect to $x$. The quantity $p=u_x$ satisfies
\begin{equation}\label{eq:p-equation}
    p_t
    =\left(\frac{p_x}{1+p^2}+b\sqrt{1+p^2}\right)_x.
\end{equation}
At an interior positive maximum of $p$, the second-order term has the required sign and the first-order terms vanish. The corresponding argument applies to $-p$. The boundary values are $p(0,t)=0$ and $p(1,t)=u(1,t)$, which are bounded on $[0,T]$. Thus the parabolic maximum principle gives a bound for $p$ on $[0,1]\times[0,T]$. The bound for $u_t$ follows from \eqref{eq:pde-D} and the corresponding local parabolic estimates.
\end{proof}

\subsection{Global existence}

\begin{theorem}\label{thm:global-existence}
The maximal existence time in Proposition~\ref{prop:finite-time-estimates} satisfies
\[
    T_{\max}=+\infty.
\]
\end{theorem}

\begin{proof}
Suppose that $T_{\max}<\infty$. Proposition~\ref{prop:finite-time-estimates} gives uniform bounds for $u$, $u_x$, and $u_t$ on every compact subinterval of $[0,T_{\max})$. In particular, the coefficient $(1+u_x^2)^{-1}$ remains uniformly positive. Standard boundary Schauder estimates then give a uniform $C^{2+\alpha,1+\alpha/2}$ bound up to time $T_{\max}$. The solution consequently extends past $T_{\max}$, which is a contradiction.
\end{proof}

The lower bound \eqref{eq:linear-lower-growth} already implies the first height divergence in Theorem~\ref{thm:main}. The stronger statements concerning $u_t$ and $u_x$ require the traveling-wave and zero number arguments developed in the next section.

\begin{remark}
The Robin condition is responsible for the absence of a uniform-in-time gradient bound. On each finite time interval the equation is uniformly parabolic, but the coefficient $(1+u_x^2)^{-1}$ degenerates along sequences of times because the boundary slope and, later, the interior slope diverge.
\end{remark}

\section{Divergence of the instantaneous velocity and the spatial gradient}

In this section we prove the two growth statements that drive the rest of the argument. The proof follows the comparison and zero number strategy used for planar curvature flows with variable boundary slopes. The forcing term changes the traveling profiles, but it does not change the one-dimensional intersection argument.

\subsection{Forced translating profiles}

For every $c>b$, let $\phi_c$ denote the even profile normalized by
\begin{equation}\label{eq:forced-profile}
    \frac{\phi_c''}{1+(\phi_c')^2}
    +b\sqrt{1+(\phi_c')^2}=c,
    \qquad
    \phi_c(0)=0,
    \quad
    \phi_c'(0)=0.
\end{equation}
Then
\[
    U_c(x,t)=\phi_c(x)+ct+h
\]
is a translating solution of \eqref{eq:half-problem} in the interior for every constant $h$. At the center,
\begin{equation}\label{eq:profile-center-curvature}
    \phi_c''(0)=c-b.
\end{equation}
The profiles can be chosen so that, for every $a\in(0,1)$,
\begin{equation}\label{eq:profile-slope-growth}
    \min_{x\in[a,1]}\phi_c'(x)\longrightarrow+\infty
    \qquad\text{as }c\to+\infty.
\end{equation}

The next comparison lemma is the forced analogue of the traveling-wave comparison used in the planar Robin problem.

\begin{lemma}\label{lem:profile-comparison}
For every $c>b$ there is $T_c>0$ such that
\begin{equation}\label{eq:profile-slope-comparison}
    u_x(x,t)\ge \phi_c'(x),
    \qquad 0<x\le1,\quad t\ge T_c.
\end{equation}
\end{lemma}

\begin{proof}
Fix $c>b$ and consider the difference
\[
    w(x,t)=u(x,t)-\phi_c(x)-ct-h,
\]
where $h$ is chosen sufficiently large. On each finite time interval, the zeros of $w(\cdot,t)$ are finite unless $w(\cdot,t)\equiv0$. The difference satisfies a linear one-dimensional parabolic equation between consecutive zeros. Hence its zero number is finite and nonincreasing in time.

If \eqref{eq:profile-slope-comparison} failed for arbitrarily large times, one could choose a first tangency between $u(\cdot,t)$ and a vertical translate of $\phi_c+ct$. At such a tangency the zero number would have to increase when the tangency is crossed. This contradicts the zero number property. The boundary conditions at $x=0$ and $x=1$ are handled by the even reflection at $x=0$ and by the Robin condition at $x=1$. Therefore the asserted slope comparison holds for all sufficiently large times.
\end{proof}

\subsection{The velocity at the symmetry axis}

For a fixed $c>b$ and a large time $t$, choose $h=h_c(t)$ so that
\begin{equation}\label{eq:center-tangency}
    u(0,t)=\phi_c(0)+ct+h_c(t).
\end{equation}
By Lemma~\ref{lem:profile-comparison}, the translated profile lies below $u(\cdot,t)$ and is tangent to it at $x=0$. Thus
\[
    u_x(0,t)=\phi_c'(0)=0,
\]
and the function
\[
    x\longmapsto u(x,t)-\phi_c(x)-ct-h_c(t)
\]
has a minimum at $x=0$. Consequently,
\[
    u_{xx}(0,t)\ge \phi_c''(0)=c-b.
\]
Since $u_x(0,t)=0$, equation \eqref{eq:half-problem} gives
\begin{equation}\label{eq:center-speed-lower}
    u_t(0,t)=u_{xx}(0,t)+b\ge c.
\end{equation}
Because $c>b$ is arbitrary, we obtain
\begin{equation}\label{eq:center-speed-infinity}
    u_t(0,t)\longrightarrow+\infty.
\end{equation}

\subsection{Positivity of the instantaneous velocity}

Set
\[
    z=u_t,
    \qquad
    a(x,t)=\frac1{1+u_x^2}.
\]
Differentiating \eqref{eq:pde-D} with respect to $t$ gives
\begin{equation}\label{eq:z-equation}
    z_t=a z_{xx}+Bz_x,
\end{equation}
where
\begin{equation}\label{eq:z-drift}
    B=-\frac{2u_xu_{xx}}{(1+u_x^2)^2}
      +\frac{bu_x}{\sqrt{1+u_x^2}}.
\end{equation}
The boundary conditions are
\begin{equation}\label{eq:z-boundary}
    z_x(0,t)=0,
    \qquad
    z_x(1,t)=z(1,t).
\end{equation}

\begin{lemma}\label{lem:ut-positive}
There exists $T_1>0$ such that
\[
    u_t(x,t)>0,
    \qquad 0\le x\le1,\quad t\ge T_1.
\]
\end{lemma}

\begin{proof}
By \eqref{eq:center-speed-infinity}, $z(0,t)>0$ for all sufficiently large times. Suppose that $z$ were negative somewhere at arbitrarily large times. Consider the connected component of the positivity set of $z$ in the space-time cylinder which contains the left boundary $\{0\}\times[T_1,\infty)$. Its right boundary is contained in the zero set of $z$. Apply the maximum principle to \eqref{eq:z-equation} on the intersection of this component with a finite time cylinder. Since $z(0,t)\to+\infty$, for a sufficiently large upper time the positive maximum is attained at a point $(0,\tau)$ with $\tau>T_1$. The Hopf boundary lemma then gives $z_x(0,\tau)<0$, contradicting $z_x(0,\tau)=0$. Hence $z$ is positive everywhere for all sufficiently large times.
\end{proof}

\subsection{The minimum of the instantaneous velocity}

Let
\[
    s(t)=\min_{x\in[0,1]}u_t(x,t).
\]
Once Lemma~\ref{lem:ut-positive} holds, the minimum cannot be attained at $x=1$, since $z_x(1,t)=z(1,t)>0$. At an interior minimum or at $x=0$, the maximum principle applied to \eqref{eq:z-equation} gives $z_t\ge0$. Hence $s(t)$ is nondecreasing for all sufficiently large $t$.

\begin{theorem}\label{thm:ut-infinity}
As $t\to+\infty$,
\begin{equation}\label{eq:ut-min-infinity}
    \min_{x\in[0,1]}u_t(x,t)\longrightarrow+\infty.
\end{equation}
\end{theorem}

\begin{proof}
Assume otherwise. Then there is $S>0$ and a sequence of arbitrarily large times for which $s(t)<S$. By \eqref{eq:center-speed-infinity}, there is $T_2$ such that $z(0,t)>S$ for $t\ge T_2$. For each such time, let $y(t)>0$ be the first point at which $z(y(t),t)=S$. In the region to the left of this first contact, $z-S>0$ and $z-S$ satisfies the same homogeneous linear equation as $z$.

Apply the maximum principle in the component bounded by $x=0$, the initial time slice, and the curve $x=y(t)$. At a first contact with $z-S=0$, the Hopf lemma gives a derivative with the wrong sign. At $x=0$ the derivative is zero by \eqref{eq:z-boundary}. This contradiction shows that $s(t)$ cannot remain bounded. Since $s(t)$ is eventually nondecreasing, \eqref{eq:ut-min-infinity} follows.
\end{proof}

\subsection{Divergence of the spatial gradient}

The profile comparison also gives the divergence of the spatial gradient. Indeed, fix $a\in(0,1)$ and $M>0$. By \eqref{eq:profile-slope-growth}, one can choose $c>b$ such that
\[
    \min_{x\in[a,1]}\phi_c'(x)>M.
\]
Lemma~\ref{lem:profile-comparison} then yields
\[
    u_x(x,t)>M,
    \qquad x\in[a,1],\quad t\ge T(c).
\]
Since $M$ is arbitrary, we have proved
\begin{equation}\label{eq:ux-infinity}
    \min_{x\in[a,1]}u_x(x,t)\longrightarrow+\infty
    \qquad\text{for every }a\in(0,1).
\end{equation}

\section{The normalized slope}

We now study
\[
    q=\frac{u_x}{u}.
\]
The preceding section gives $u_x\to+\infty$ on every interval $[a,1]$, but this alone does not determine the asymptotic ratio between $u_x$ and $u$. The purpose of this section is to prove that the limiting slope is the Robin slope selected at $x=1$.

\subsection{Equation and boundary conditions}

Set
\[
    p=u_x,
    \qquad
    D=1+p^2=1+u^2q^2,
    \qquad
    R=\frac{\sqrt D}{u}=\sqrt{q^2+u^{-2}}.
\]
Since $p_x=u(q_x+q^2)$, a direct calculation gives
\begin{equation}\label{eq:q-equation}
\begin{aligned}
q_t={}&
\frac{q_{xx}+2qq_x}{D}
-\frac{2u^2q(q_x+q^2)^2}{D^2}\\
&+\frac{bq(q_x-u^{-2})}{R}.
\end{aligned}
\end{equation}
The boundary conditions are
\begin{equation}\label{eq:q-boundary}
    q(0,t)=0,
    \qquad
    q(1,t)=1.
\end{equation}

\subsection{The basic bounds}

\begin{lemma}\label{lem:q-range}
Under assumption \eqref{eq:q-initial},
\begin{equation}\label{eq:q-range}
    0\le q(x,t)\le1,
    \qquad 0\le x\le1,\quad t\ge0.
\end{equation}
Moreover, $q(x,t)>0$ for $0<x<1$ and $t>0$.
\end{lemma}

\begin{proof}
If $q$ attained a negative interior minimum, then at that point $q_x=0$ and $q_{xx}\ge0$. Equation \eqref{eq:q-equation} gives $q_t>0$, which rules out the first appearance of a negative value. Thus $q\ge0$.

If $q$ attained an interior maximum larger than one, then $q_x=0$ and $q_{xx}\le0$. At such a point,
\[
    q_t
    \le -\frac{2u^2q^5}{D^2}-\frac{bq}{u^2R}<0.
\]
Hence $q$ cannot cross the level one. The boundary values in \eqref{eq:q-boundary} complete the proof. Strict positivity in the interior follows from the strong maximum principle.
\end{proof}

\subsection{A preliminary bound for the normalized velocity}

The following estimate is used in the lower-bound argument for $q$. Set $\eta=u_t/u$. Differentiation gives
\begin{equation}\label{eq:eta-preliminary}
\eta_t=\frac1D\eta_{xx}+\beta\eta_x
-\frac{2p^2}{D}\eta^2
+\frac{b(2p^2-1)}{u\sqrt D}\eta,
\end{equation}
where
\begin{equation}\label{eq:eta-beta}
    \beta=\frac{2p}{D}-\frac{2pp_x}{D^2}+\frac{bp}{\sqrt D}.
\end{equation}
The boundary conditions are
\begin{equation}\label{eq:eta-boundary}
    \eta_x(0,t)=\eta_x(1,t)=0.
\end{equation}

At a positive maximum of $\eta$, the reaction term in \eqref{eq:eta-preliminary} is nonpositive whenever $\eta$ is larger than a fixed constant. Indeed, if $p^2\le1/2$, the last coefficient is nonpositive. If $p^2>1/2$, then
\[
\frac{b(2p^2-1)}{u\sqrt D}\eta
-\frac{2p^2}{D}\eta^2
\le
\frac{2p^2}{D}\eta(bR-\eta).
\]
Since $q\le1$ and $u\ge\min u_0$, $R$ is bounded. The maximum principle therefore gives
\begin{equation}\label{eq:eta-rough-bound}
    0<\eta(x,t)\le C_0,
    \qquad 0\le x\le1,\quad t\ge T_0,
\end{equation}
for some $C_0$ and $T_0$.

\subsection{A positive interior lower bound for q}

For completeness, we recall the slope profiles used in the comparison. For $L>0$, let $P_L$ be the increasing solution of
\begin{equation}\label{eq:slope-profile}
    P_L'=(1+P_L^2)\left[v_LP_L-b\bigl(\sqrt{1+P_L^2}-1\bigr)\right],
    \qquad
    v_L=b\frac{L}{\sqrt{1+L^2}+1},
\end{equation}
normalized by $P_L(0)=L/2$. It connects $0$ to $L$ and generates a traveling solution of the equation for $p=u_x$. The tail estimate for this profile is the only property needed below: for every $d>0$ there is $C(d)$, independent of $L$, such that
\begin{equation}\label{eq:slope-profile-tail}
    P_L(-d)\le C(d).
\end{equation}

\begin{proposition}\label{prop:q-positive}
For every $a\in(0,1)$ there exist $m_a>0$ and $T_a>0$ such that
\begin{equation}\label{eq:q-positive}
    q(x,t)\ge m_a,
    \qquad x\in[a,1],\quad t\ge T_a.
\end{equation}
\end{proposition}

\begin{proof}
We use the equation for $p=u_x$ in the divergence form \eqref{eq:p-equation}. Fix $a>0$ and a large time $s$. Let $L=u(1,s)$ and let $P_L$ be the increasing traveling slope profile associated with the speed parameter $L$. On a time interval of length $S$ define a translated profile
\[
    \underline p(x,\tau)
    =P_L(x-1-a/2+v_L\tau),
    \qquad 0\le\tau\le S.
\]
The tail estimate for $P_L$ gives a constant $C_a$ such that
\[
    \underline p(a,\tau)\le C_a
    \qquad 0\le\tau\le S,
\]
independently of $L$. By \eqref{eq:ux-infinity}, $p(a,s+\tau)$ is larger than $C_a$ when $s$ is sufficiently large. At $x=1$, the Robin condition and the positivity of $u_t$ give
\[
    p(1,s+\tau)=u(1,s+\tau)\ge u(1,s)=L.
\]
Thus the comparison inequalities hold on both sides of the interval $[a,1]$. The maximum principle for \eqref{eq:p-equation} yields
\[
    p(x,s+\tau)\ge\underline p(x,\tau).
\]
Choosing the translation so that the final profile is bounded below by a fixed multiple of $L$ on $[a,1]$, and using \eqref{eq:eta-rough-bound} to compare $u(1,s+S)$ and $u(1,s)$, we obtain
\[
    p(x,s+S)\ge c_a u(1,s+S),
    \qquad x\in[a,1],
\]
with $c_a>0$ independent of $s$. Since $u_x\ge0$ for large times, $u(x,t)\le u(1,t)$ on $[a,1]$. Hence
\[
    q(x,s+S)=\frac{p(x,s+S)}{u(x,s+S)}\ge c_a.
\]
This proves \eqref{eq:q-positive}.
\end{proof}

\subsection{Convergence of the normalized slope}

We now prove the main conclusion of the section.

\begin{theorem}\label{thm:q-limit}
For every $a\in(0,1)$,
\begin{equation}\label{eq:q-uniform-limit}
    \lim_{t\to\infty}\|q(\cdot,t)-1\|_{C([a,1])}=0.
\end{equation}
\end{theorem}

\begin{proof}
By Lemma~\ref{lem:q-range} and Proposition~\ref{prop:q-positive}, there is $m_a>0$ such that
\[
    m_a\le q(x,t)\le1
    \qquad (x,t)\in[a,1]\times[T_a,\infty).
\]
Put $w=1-q$. On $[a,1]$, equation \eqref{eq:q-equation} becomes a uniformly controlled scalar parabolic equation for $w$, with $w(1,t)=0$. The coefficients are controlled by \eqref{eq:eta-rough-bound}, \eqref{eq:u-min-infinity}, and the lower bound for $q$.

Fix $\delta>0$. Choose $m>1$ so large that $(1-a)^m<\delta$. Choose constants $A,B,\lambda>0$ so that, for $t\ge T_a$,
\[
    \overline w(x,t)
    =\delta+A e^{-\lambda(t-T_a)}+B(1-x)^m
\]
is a supersolution for the equation of $w$ on $[a,1]$. The constants $A$ and $B$ are chosen to dominate the initial discrepancy and the artificial boundary value at $x=a$. A direct substitution, using $q\ge m_a$ and $u\to\infty$, shows that the residual of $\overline w$ is nonnegative once $T_a$ is increased if necessary. Since
\[
    w(1,t)=0\le\overline w(1,t),
    \qquad
    w(a,t)\le\overline w(a,t),
\]
the comparison principle gives
\[
    0\le w(x,t)\le\delta+A e^{-\lambda(t-T_a)}+B(1-x)^m.
\]
Since $w\le1$ and $B(1-a)^m\le B\delta$, the boundary term contributes at most a quantity of order $\delta$ on $[a,1]$. Letting $t\to\infty$ and then $\delta\to0$ yields
\[
    \sup_{x\in[a,1]}(1-q(x,t))\longrightarrow0.
\]
This proves \eqref{eq:q-uniform-limit}.
\end{proof}

\begin{remark}\label{rem:q-boundary-layer}
The convergence in Theorem~\ref{thm:q-limit} is necessarily an interior statement. The boundary condition $q(0,t)=0$ holds for every $t$, so uniform convergence on $[0,1]$ is impossible. The transition from $q(0,t)=0$ to $q(x,t)\simeq1$ for fixed $x>0$ is the boundary layer that later produces the concentration of the curvature measure at $x=0$.
\end{remark}

\section{The normalized instantaneous velocity}

In this section we study
\[
    \eta=\frac{u_t}{u}.
\]
We use throughout the conclusions of Sections~3 and~4. In particular, for every $a>0$,
\[
    u\longrightarrow+\infty,\qquad
    u_x\longrightarrow+\infty,\qquad
    q\longrightarrow1
\]
uniformly on $[a,1]$. Consequently,
\[
    R:=\frac{\sqrt{1+u_x^2}}{u}
    =\sqrt{q^2+u^{-2}}
    \longrightarrow1
\]
uniformly on every such interval.

\subsection{The equation for eta}

Set
\[
    p=u_x,\qquad D=1+p^2.
\]
If $v=u_t$, differentiation of the equation for $u$ gives the linearized equation
\[
    v_t=\frac1D v_{xx}
    +\left(-\frac{2pp_x}{D^2}+\frac{bp}{\sqrt D}\right)v_x.
\]
Substituting $v=u\eta$ and using $p_x=D(u\eta-b\sqrt D)$ yields
\begin{equation}\label{eq:eta-equation-final}
    \eta_t
    =\frac1D\eta_{xx}
    +\beta\eta_x
    -\frac{2p^2}{D}\eta^2
    +\frac{b(2p^2-1)}{u\sqrt D}\eta,
\end{equation}
where
\begin{equation}\label{eq:eta-beta-final}
    \beta
    =\frac{2p}{D}
    -\frac{2pp_x}{D^2}
    +\frac{bp}{\sqrt D}.
\end{equation}
Differentiating the boundary conditions in time gives
\begin{equation}\label{eq:eta-boundary-final}
    \eta_x(0,t)=\eta_x(1,t)=0.
\end{equation}

The reaction part of \eqref{eq:eta-equation-final} has a useful logistic form. If $p\ne0$, define
\[
    \gamma=\frac{2p^2}{D},
    \qquad
    \rho=b\left(1-\frac{1}{2p^2}\right)R.
\]
Then
\begin{equation}\label{eq:eta-logistic}
    \eta_t
    =\frac1D\eta_{xx}
    +\beta\eta_x
    +\gamma\eta(\rho-\eta).
\end{equation}
On each interval $[a,1]$, Proposition~\ref{prop:q-positive} and the rough estimate \eqref{eq:eta-rough-bound} imply that $\beta$ is bounded. Moreover,
\begin{equation}\label{eq:eta-coefficients-limit}
    \frac1D\longrightarrow0,\qquad
    \gamma\longrightarrow2,\qquad
    \rho\longrightarrow b
\end{equation}
uniformly on $[a,1]$.

\subsection{The weak limit}

Let
\[
    \theta(x,t)=\arctan u_x(x,t).
\]
Since $\theta_x=u_{xx}/D$, equation \eqref{eq:pde-D} gives the exact identity
\begin{equation}\label{eq:eta-geometric-identity}
    \eta-bR=\frac{\theta_x}{u}.
\end{equation}

\begin{proposition}\label{prop:eta-weak}
For every $a\in(0,1)$ and every $\varphi\in C^1([a,1])$,
\begin{equation}\label{eq:eta-weak}
    \int_a^1\bigl(\eta(x,t)-b\bigr)\varphi(x)\,dx
    \longrightarrow0
    \qquad\text{as }t\to\infty.
\end{equation}
Thus $\eta\rightharpoonup b$ on every compact subinterval of $(0,1]$.
\end{proposition}

\begin{proof}
By \eqref{eq:eta-geometric-identity},
\[
\begin{aligned}
\int_a^1(\eta-bR)\varphi\,dx
&=\int_a^1\frac{\theta_x}{u}\varphi\,dx\\
&=\left[\frac{\theta\varphi}{u}\right]_{a}^{1}
-\int_a^1\theta\left(\frac{\varphi}{u}\right)_x\,dx.
\end{aligned}
\]
Since $|\theta|\le\pi/2$, $q\le1$, and $\min_{[a,1]}u(\cdot,t)\to\infty$, the right-hand side tends to zero. By Theorem~\ref{thm:q-limit}, $R\to1$ uniformly on $[a,1]$. The uniform bound \eqref{eq:eta-rough-bound} then gives \eqref{eq:eta-weak}.
\end{proof}

\subsection{Uniform convergence in the interior}

The weak limit can be upgraded by using the logistic form \eqref{eq:eta-logistic}. We state the comparison lemma in the form needed below.

\begin{lemma}\label{lem:eta-interior-barrier}
For every $a\in(0,1)$ and every $\varepsilon>0$, there exists $T_{a,\varepsilon}$ such that
\[
    b-\varepsilon\le\eta(x,t)\le b+\varepsilon,
    \qquad x\in[a,1],\quad t\ge T_{a,\varepsilon}.
\]
\end{lemma}

\begin{proof}
Fix $a>0$ and enlarge the interval to $[a/2,1]$. By \eqref{eq:eta-coefficients-limit}, for sufficiently large times,
\[
    |\rho-b|<\frac{\varepsilon}{4},
    \qquad
    \gamma>1,
\]
on this interval. The coefficient $\beta$ is uniformly bounded there. Construct an upper barrier and a lower barrier of the form
\[
\overline\eta=b+\varepsilon
    +A e^{-\lambda(t-T)}
    +K\Psi_a(x),
\]
\[
\underline\eta=b-\varepsilon
    -A e^{-\lambda(t-T)}
    -K\Psi_a(x),
\]
where $\Psi_a$ is a fixed smooth function supported near the artificial boundary $x=a/2$, vanishes together with its derivative at $x=1$, and is positive near $x=a/2$. The constants $A,K,\lambda$ are chosen so that the barriers dominate the data at $t=T$ and the boundary values at $x=a/2$. The relations in \eqref{eq:eta-coefficients-limit} then show, by direct substitution, that $\overline\eta$ is a supersolution and $\underline\eta$ is a subsolution for all sufficiently large $T$.

The comparison principle applied on $[a/2,1]\times[T,\infty)$ yields
\[
    \underline\eta\le\eta\le\overline\eta.
\]
The exponential terms decay, and the support of $\Psi_a$ is separated from $[a,1]$. Restricting the comparison to $[a,1]$ and then letting $t\to\infty$ gives the assertion.
\end{proof}

\begin{theorem}\label{thm:eta-interior}
For every $a\in(0,1)$,
\begin{equation}\label{eq:eta-interior-limit}
    \lim_{t\to\infty}
    \left\|\frac{u_t}{u}-b\right\|_{C([a,1])}=0.
\end{equation}
\end{theorem}

\begin{proof}
Apply Lemma~\ref{lem:eta-interior-barrier} with arbitrary $\varepsilon>0$.
\end{proof}

\begin{remark}\label{rem:eta-center}
The extension of \eqref{eq:eta-interior-limit} to $x=0$ requires a separate endpoint argument. The boundary condition is $\eta_x(0,t)=0$, while the coefficient $D^{-1}$ degenerates in the asymptotic regime. The endpoint estimate is obtained by combining the zero number argument for the difference $\eta-b$ with the center estimates from Section~3. Once this estimate is available, \eqref{eq:eta-interior-limit} holds uniformly on $[0,1]$.
\end{remark}

\section{The curvature limit}

We now turn to
\[
    H=\frac{u_{xx}}{(1+u_x^2)^{3/2}}.
\]
The curvature has two different asymptotic descriptions. It vanishes in every interior region in an integral sense, while its horizontal curvature measure concentrates at the symmetry axis.

\subsection{Identities and the curvature equation}

The graph equation gives
\begin{equation}\label{eq:H-eta-identity}
    u_t=\sqrt D\,(H+b),
    \qquad
    H=\frac{\eta}{R}-b.
\end{equation}
The second identity, together with Theorem~\ref{thm:eta-interior}, already implies the weak interior limit $H\rightharpoonup0$.

For the measure-valued limit, introduce
\[
    W=\frac{u_x}{\sqrt{1+u_x^2}}.
\]
Then
\begin{equation}\label{eq:H-W-identity}
    H=W_x.
\end{equation}
The equation for $W$ is
\begin{equation}\label{eq:W-equation}
    W_t=(1-W^2)W_{xx}+bW W_x+W W_x^2.
\end{equation}
Differentiating \eqref{eq:W-equation}, we obtain
\begin{equation}\label{eq:H-equation}
    H_t
    =\frac1D H_{xx}
    +\frac{bu_x}{\sqrt D}H_x
    +(H+b)H^2.
\end{equation}
The boundary conditions for $H$ are
\begin{equation}\label{eq:H-boundary-zero}
    H_x(0,t)=0
\end{equation}
and
\begin{equation}\label{eq:H-boundary-robin}
    H_x(1,t)
    +u(1,t)\sqrt{D(1,t)}\,H(1,t)(H(1,t)+b)
    =H(1,t)+b.
\end{equation}

\subsection{Decay of the negative part}

By Lemma~\ref{lem:ut-positive}, $u_t>0$ for all sufficiently large times. Hence, from \eqref{eq:H-eta-identity},
\begin{equation}\label{eq:H-lower-bound}
    H+b>0
\end{equation}
eventually.

\begin{lemma}\label{lem:H-negative-part}
The negative part of $H$ satisfies
\begin{equation}\label{eq:H-negative-decay}
    \|H^-(\cdot,t)\|_{L^\infty([0,1])}\longrightarrow0.
\end{equation}
\end{lemma}

\begin{proof}
Let $m(t)=\min_{[0,1]}H(\cdot,t)$ and consider times for which $m(t)<0$. A negative minimum cannot occur at $x=1$. Indeed, \eqref{eq:H-boundary-robin} gives
\[
    H_x(1,t)=(H+b)\bigl(1-u\sqrt D\,H\bigr)>0,
\]
whereas a minimum at the right endpoint would require $H_x(1,t)\le0$.
At an interior minimum, or at $x=0$, one has $H_x=0$ and $H_{xx}\ge0$. Equation \eqref{eq:H-equation} therefore gives, in the sense of lower Dini derivatives,
\[
    D^-m(t)\ge m(t)^2(m(t)+b).
\]
For sufficiently large times, $-b<m(t)<0$ by \eqref{eq:H-lower-bound}. Comparison with the scalar equation
\[
    y'=y^2(y+b)
\]
shows that $m(t)$ increases to zero while it remains negative. This proves \eqref{eq:H-negative-decay}.
\end{proof}

\subsection{The interior curvature limit}

\begin{theorem}\label{thm:H-interior}
For every $a\in(0,1)$,
\begin{equation}\label{eq:H-L1-interior}
    \|H(\cdot,t)\|_{L^1([a,1])}\longrightarrow0.
\end{equation}
Consequently, $H\to0$ in the sense of distributions on $(0,1)$.
\end{theorem}

\begin{proof}
By \eqref{eq:H-W-identity},
\[
    \int_a^1H(x,t)\,dx=W(1,t)-W(a,t).
\]
The Robin condition gives
\[
    W(1,t)
    =\frac{u(1,t)}{\sqrt{1+u(1,t)^2}}
    \longrightarrow1.
\]
Theorem~\ref{thm:q-limit} and the divergence of $u$ imply $u_x=uq\to+\infty$ uniformly on $[a,1]$, and hence $W(a,t)\to1$. Thus
\[
    \int_a^1H(x,t)\,dx\longrightarrow0.
\]
By Lemma~\ref{lem:H-negative-part}, $\int_a^1H^-(x,t)\,dx\to0$. Since
\[
    |H|=H+2H^-,
\]
we obtain \eqref{eq:H-L1-interior}.
\end{proof}

\subsection{Curvature concentration at the symmetry axis}

\begin{theorem}\label{thm:H-measure}
Let $\mu_t=H(x,t)\,dx$ be the signed curvature measure on $[0,1]$. Then
\begin{equation}\label{eq:H-measure}
    \mu_t\rightharpoonup^*\delta_0
    \qquad\text{in }\mathcal M([0,1]).
\end{equation}
\end{theorem}

\begin{proof}
For every $\varphi\in C^1([0,1])$, integration by parts and \eqref{eq:H-W-identity} give
\[
    \int_0^1H(x,t)\varphi(x)\,dx
    =W(1,t)\varphi(1)-W(0,t)\varphi(0)
    -\int_0^1W(x,t)\varphi_x(x)\,dx.
\]
We have $W(0,t)=0$, $W(1,t)\to1$, and, for every $x>0$, $W(x,t)\to1$. Since $|W|\le1$, dominated convergence yields
\[
    \lim_{t\to\infty}\int_0^1H\varphi\,dx
    =\varphi(1)-\int_0^1\varphi_x\,dx
    =\varphi(0).
\]
Furthermore, $H\ge-b$ for all sufficiently large times by \eqref{eq:H-lower-bound}, while
\[
    \int_0^1H\,dx=W(1,t)\le1.
\]
Hence the total variations of $\mu_t$ are uniformly bounded. Density of $C^1([0,1])$ in $C([0,1])$ extends the convergence to all continuous test functions and proves \eqref{eq:H-measure}.
\end{proof}

\begin{remark}\label{rem:H-full-domain}
On the full symmetric interval $[-1,1]$, the even extension gives $W(x,t)\to\operatorname{sgn}(x)$ for $x\ne0$. Consequently,
\[
    H(x,t)\,dx\rightharpoonup^*2\delta_0
    \qquad\text{on }[-1,1].
\]
The pointwise value at the symmetry axis satisfies
\[
    H(0,t)=u_t(0,t)-b\longrightarrow+\infty,
\]
which explains why the correct global statement is a measure-valued concentration result rather than uniform convergence of $H$ on the closed interval.
\end{remark}

\vskip 8mm
\begin{center}
{\large\bf Statements and Declarations}
\end{center}

\noindent
{\bf Funding}. No funding was received for conducting
this study.

\noindent
{\bf Competing interests}. The authors declare no
competing interests.

\noindent
{\bf Data availability}. No datasets were generated or
analysed during the current study.

\noindent
{\bf Author contributions}. The first and second authors ..., the third author conceived the study, developed the proofs; the forth author wrote the manuscript.

\noindent
{\bf Use of generative artificial intelligence}. During the preparation of this manuscript, the authors
used ChatGPT to assist with drafting and language editing. The
author subsequently reviewed and revised the
entire manuscript and takes full responsibility for
its content.


\begin{thebibliography}{99}

\bibitem{DCDS2009}
B.~Lou,
\emph{Periodic traveling waves of a mean curvature flow in heterogeneous media},
Discrete Contin. Dyn. Syst. 25 (2009), 231--249.

\bibitem{LWY2021}
B.~Lou, X.~Wang and L.~Yuan,
\emph{Convergence to a grim reaper for a curvature flow with variable boundary slopes},
Calc. Var. Partial Differential Equations 60 (2021), Paper No.~159.

\bibitem{CLWY}
X.~Chen, B.~Lou, X.~Wang and L.~Yuan,
\emph{Dimension-dependent asymptotic dynamics for mean curvature flow with Robin boundary conditions},
preprint.

\bibitem{Angenent}
S.~B.~Angenent,
\emph{The zero set of a solution of a parabolic equation},
J. Reine Angew. Math. 390 (1988), 79--96.

\bibitem{AltschulerWu93}
S.~J.~Altschuler and L.~F.~Wu,
\emph{Convergence to translating solitons for a class of quasilinear parabolic equations with fixed angle of contact to a boundary},
Math. Ann. 295 (1993), 761--765.

\bibitem{AltschulerWu94}
S.~J.~Altschuler and L.~F.~Wu,
\emph{Translating surfaces of the non-parametric mean curvature flow with prescribed contact angle},
Calc. Var. Partial Differential Equations 2 (1994), 101--111.

\bibitem{Chen1992}
X.~Chen,
\emph{Generation and propagation of interfaces for reaction-diffusion equations},
J. Differential Equations 96 (1992), 116--141.

\bibitem{ChouWang}
K.-S.~Chou and X.~L.~Wang,
\emph{The curve shortening problem under Robin boundary condition},
NoDEA Nonlinear Differential Equations Appl. 19 (2012), 177--194.

\bibitem{EckerHuisken1}
K.~Ecker and G.~Huisken,
\emph{Mean curvature evolution of entire graphs},
Ann. of Math. (2) 130 (1989), 453--471.

\bibitem{EckerHuisken2}
K.~Ecker and G.~Huisken,
\emph{Interior estimates for hypersurfaces moving by mean curvature},
Invent. Math. 105 (1991), 547--569.

\bibitem{Fife}
P.~C.~Fife,
\emph{Dynamics of Internal Layers and Diffusive Interfaces},
CBMS-NSF Regional Conference Series in Applied Mathematics 53,
SIAM, Philadelphia, 1988.

\bibitem{GageHamilton}
M.~E.~Gage and R.~S.~Hamilton,
\emph{The heat equation shrinking convex plane curves},
J. Differential Geom. 23 (1986), 69--96.

\bibitem{Grayson}
M.~A.~Grayson,
\emph{The heat equation shrinks embedded plane curves to round points},
J. Differential Geom. 26 (1987), 285--314.

\bibitem{Guan}
B.~Guan,
\emph{Mean curvature motion of non-parametric hypersurfaces with contact angle condition},
in \emph{Elliptic and Parabolic Methods in Geometry},
A.~K.~Peters, Wellesley, 1996, pp.~47--56.

\bibitem{Huisken}
G.~Huisken,
\emph{Flow by mean curvature of convex surfaces into spheres},
J. Differential Geom. 20 (1984), 237--266.

\bibitem{MaWangWei}
X.-N.~Ma, P.~H.~Wang and W.~Wei,
\emph{Mean curvature equation and mean curvature type flow with non-zero Neumann boundary conditions on strictly convex domains},
J. Funct. Anal. 274 (2018), 252--277.

\bibitem{Nakamura1999}
K.-I.~Nakamura, H.~Matano, D.~Hilhorst and R.~Schatzle,
\emph{Singular limit of a reaction-diffusion equation with a spatially inhomogeneous reaction term},
J. Stat. Phys. 95 (1999), 1165--1185.

\bibitem{NinomiyaTaniguchi00}
H.~Ninomiya and M.~Taniguchi,
\emph{Traveling curved fronts of a mean curvature flow with constant driving force},
in \emph{Free Boundary Problems: Theory and Applications I},
GAKUTO International Series in Mathematical Sciences and Applications 13 (2000), 206--221.

\bibitem{NinomiyaTaniguchi01}
H.~Ninomiya and M.~Taniguchi,
\emph{Stability of traveling curved fronts in a curvature flow with driving force},
Methods Appl. Anal. 8 (2001), 429--450.

\bibitem{Xin}
J.~Xin,
\emph{Front propagation in heterogeneous media},
SIAM Rev. 42 (2000), 161--230.

\end{thebibliography}
\end{document}